\documentclass[11pt]{amsart}
\usepackage{amsmath,amsthm,amssymb,enumerate,mathscinet,mathtools}
\usepackage{fullpage}
\usepackage{graphicx,xcolor,pgfplots}
\usepackage{verbatim}
\usepackage{mathrsfs}
\usepackage{hyperref}
\usepackage{enumitem}
\usepackage{blindtext}
\usepackage{multicol}
\newcommand{\fm}{f_{\max}}

\newtheorem{theorem}{Theorem}[section]
\newtheorem{fact}[theorem]{Fact}
\newtheorem{lemma}[theorem]{Lemma}
\newtheorem{question}[theorem]{Question}

\newtheorem{conjecture}[theorem]{Conjecture}
\newtheorem{claim}[theorem]{Claim}
\newtheorem{definition}[theorem]{Definition}

\newtheorem{thm}[theorem]{Theorem}
\newtheorem{prop}[theorem]{Proposition}

\newtheorem{remark}[theorem]{Remark}

\newtheorem{remarks}[theorem]{Remarks}

\numberwithin{equation}{section}

\def\eps{\varepsilon}
\def\N{\mathbb{N}}
\def\Z{\mathbb{Z}}
\def\R{\mathbb{R}}
\def\fms{f^\star_{\max}}
\newcommand{\mis}{\textsc{mis}}
\makeatletter
\newcommand*{\rom}[1]{\expandafter\@slowromancap\romannumeral #1@}
\makeatother

\def\COMMENT#1{}
\let\COMMENT=\footnote

\allowdisplaybreaks

\title{Notes on sum-free sets in abelian groups}

\author{Nathana\"el Hassler and
 Andrew Treglown}

\thanks{NH: Université Bourgogne Europe, Dijon, France, {\tt nathanael.hassler@ens-rennes.fr}.
\\ \indent AT: University of Birmingham, United Kingdom, {\tt a.c.treglown@bham.ac.uk}. Research supported by EPSRC\\
\indent grants EP/V002279/1 and EP/V048287/1.}

\begin{document}
\begin{abstract}
In this paper we highlight a few open problems concerning maximal sum-free sets in abelian groups. 
In addition, for most even order abelian groups $G$ we asymptotically determine the number of maximal distinct sum-free subsets in  $G$. Our proof makes use  of the  container method.
\end{abstract}

\date{\today}

\maketitle

\section{Introduction}
Let $(G,+)$ be an abelian group. 
A triple $x,y,z \in G$ is a \emph{Schur triple} if $x+y=z$; if additionally $x,y,z$ are distinct, we call them a \emph{distinct Schur triple}. A subset $S \subseteq G$ is a \emph{sum-free set} if $S$ does not contain any Schur triple. Similarly, we say $S$ is a \emph{distinct sum-free set} if $S$ does not contain any distinct Schur triple. A sum-free set $S\subseteq G$ is \emph{maximal} if $S$ is not properly contained in another sum-free subset of $G$; we define the notion of a \emph{maximal distinct sum-free subset} analogously. We let $\mu(G)$ denote the size of the largest sum-free subset of $G$ and $\mu^\star(G) $ denote the size of the largest distinct sum-free subset of $G$.

The study of $\mu(G) $ dates back to the 1960s~\cite{yan} and we now (through work of Green and Ruzsa~\cite{GR-g}) have a complete understanding of the exact value of $\mu(G)$ for all abelian groups. To articulate this behaviour we need the following definition.
\begin{definition}
Let $G$ be an abelian group of order $n$.
\begin{itemize}
    \item Let $n$ be divisible by a prime $p\equiv 2 \pmod{3}$. Given the smallest such $p$, we  say that $G$ is \emph{type \rom{1}$(p)$}.
    \item If $n$ is not divisible by any prime $p\equiv 2\pmod{3}$, but $3|n$, then we say that $G$ is \emph{type \rom{2}}.
    \item Otherwise, $G$ is \emph{type \rom{3}}.
\end{itemize}
\end{definition}
\begin{theorem}\cite{yan, GR-g}\label{thmcharacter}
Given any finite abelian group $G$, if $G$ is type \rom{1}$(p)$ then $\mu(G)=|G|\left(\frac{1}{3}+\frac{1}{3p}\right)$. Otherwise, if $G$ is type \rom{2} then $\mu(G)=\frac{|G|}{3}$. Finally, if $G$ is type \rom{3} then $\mu(G)=|G|\left(\frac{1}{3}-\frac{1}{3m}\right)$ where $m$ is the exponent (largest order of any element) of $G$.
\end{theorem}
The case of type \rom{1} and \rom{2} groups in Theorem~\ref{thmcharacter} is due to Diananda and Yap~\cite{yan}; the case of type \rom{3} groups is due to Green and Ruzsa~\cite{GR-g}. 
Notice that Theorem~\ref{thmcharacter} tells us that  for every finite abelian group $G$ of order $n$, we have that
$2n/7 \leq \mu (G)\leq n/2$.

In recent years there has been significant interest  in the  study of transferring (combinatorial) theorems into the random setting; see, e.g., the  survey of Conlon~\cite{conlonsurvey}.
The next result provides a random analogue of Theorem~\ref{thmcharacter}. Note that 
it is implicit in the literature (e.g., 
it is  a simple corollary of Theorem~5.2 from~\cite{bms} and a removal lemma of Green~\cite[Theorem 1.5]{G-R}; see also~\cite{conlongowers, saxton, schacht}).

Given a finite abelian group 
$G_n$,
let $G_{n,p}$ be a random subset of $G_n$ obtained by including each element of $G_n$ independently with probability $p$. Note that throughout the paper $\mathbb N$ 
denotes the set of positive integers (i.e., it does not contain $0$).
\begin{theorem}\label{thmrandom}
Let $ (G_n)_{n \in \mathbb N}$ be a sequence of finite abelian groups where $|G_i|<|G_j|$ and $\mu (G_i)/|G_i|=\mu (G_j)/|G_j|$
for all $i<j$. 
For any  $\eps >0$ there exists  $C>0$ such that the following holds.  If $p_n\geq C|G_n|^{-1/2}$ for each $n \in \mathbb N$,
    then
    $$\lim _{n \rightarrow \infty } \mathbb P (\text{the largest sum-free set in $G_{n,p_n}$ has size } (1\pm \eps)\mu (G_n)\cdot p_n) =1.$$ \qed
\end{theorem}

As an illustrative example, a natural class to take in Theorem~\ref{thmrandom} is
$G_n:= \mathbb Z_q ^n= \mathbb Z_q \oplus \dots \oplus\mathbb Z_q$ for some fixed prime $q$. In particular, Theorem~\ref{thmcharacter}
implies that $\mu (G_i)/|G_i|=\mu (G_j)/|G_j|$
for all $i,j\in \mathbb N$.

Let $M$ be a largest sum-free subset of a finite abelian group $G_n$. Note that 
$\mathbb E (|G_{n,p} \cap M|)=\mu (G_n) \cdot p$. Thus,
 Theorem~\ref{thmrandom} can be interpreted as follows: in a `typical' subset $S$ of $G_n$ of size significantly more than $\sqrt{|G_n|}$, 
 the largest sum-free subset of $S$ has size close to what is `expected', i.e., close to $|S\cap M|$.
 Furthermore, it is not difficult to show that the bound on the probability $p_n$ in Theorem~\ref{thmrandom} is essentially best possible. 

 Note that there are also related results on  the structure of maximum-size sum-free subsets of a random subset of an abelian group $G$; see~\cite{b2, bush}.
 

\section{Maximal sum-free subsets}\label{sec2}
There has  been interest in the number of sum-free subsets $f(G)$ of a finite abelian group $G$. By considering all possible subsets of a largest sum-free subset in $G$, one obtains that $f(G) \geq 2^{\mu (G)}$. It turns out this trivial bound is not far from being tight. Indeed, Green and Ruzsa~\cite{GR-g} proved that $f(G)=2^{\mu(G)+o(n)}$ for all abelian groups $G$ of order $n$. A refined version of this theorem for type \rom{1} groups was obtained by Alon, Balogh, Morris and Samotij~\cite{alon}. Specifically, for such type \rom{1} groups they asymptotically determine the number of sum-free sets of a given size $m$, for all not too small $m$.

We now consider the \emph{maximal} sum-free version of this problem.  That is, we consider the number of maximal sum-free subsets $f_{\max}(G)$ in a finite abelian group $G$. This problem is wide open in general, though there now are several results on this topic. Improving an earlier bound of Wolfovitz~\cite{wolf}, in~\cite{BLST2} it was proven that $f_{\max} (G) \leq 3^{\mu(G)/3+o(n)}$ for all abelian groups $G$ of order $n$; in particular, this shows that $f_{\max}(G)\ll f(G)$ for all sufficiently large abelian groups. Moreover, Liu and Sharifzadeh~\cite{ls} showed that $f_{\max}(G)= 3^{\mu(G)/3+o(n)}$ for all type \rom{2} groups $G$ whose order $n$ is divisible by $9$. In contrast, there are groups where this upper bound is far from tight. In~\cite{BLST2} it was shown that $f_{\max}(\mathbb Z^k_2)=2^{\mu(G)/2+o(n)}$, and this bound was further refined in~\cite{HT} (see Theorem~\ref{thm23} below). Furthermore, in~\cite{ls} it was shown that there is a $c>0$ such that almost all even order abelian groups $G$ satisfy $f_{\max}(G) \leq 2^{(1/2-c)\mu (G)}$. Constructions from~\cite[Proposition 5.7]{BLST2} and~\cite[Proposition 5.3]{HT} show that $f_{\max} (G)\geq 2^{\mu(G)/2-2}$ for all type \rom{3} abelian groups whose exponent is 7, 13 or 19. See~\cite{BLST2, HT, ls} for further results on $f_{\max}(G)$.

These results paint a rather mixed picture, and it remains unclear how $f_{\max}(G)$ behaves in general. However, the work in~\cite{BLST2, ls} implicitly raises the following question.
\begin{question}\label{quesX}
    Given an abelian group $G$ of order $n$, is it true that 
    either $f_{\max}(G)= 3^{\mu(G)/3+o(n)}$ or $f_{\max}(G)\leq  2^{\mu (G)/2+o(n)}$?
\end{question}

{\noindent \bf Additional note.} Since we submitted this paper, Question~\ref{quesX} has been answered in the negative by Balogh, Garcia, Liu and Yang~\cite{bgly}.
In fact, they have tackled a couple of the other open problems stated in this paper; we further discuss this at the relevant points.

\smallskip

For proving upper bounds on $f_{\max}(G)$ there is now a somewhat well-trodden approach that makes use of the \emph{container method}; see~\cite{BLST, BLST2, HT, ls}. We will use this approach later in the paper also. Using this approach, in~\cite{HT} we proved the following sharp result.
\begin{theorem}\cite{HT} \label{thm23}
  If $k \in \mathbb N$ and $n:=2^k$, then
$\fm (\mathbb Z^k_2)= \left (\binom{n-1}{2}+o(1) \right ) 2^{n/4}.$ Further,
if $k\in\N$ and $n:=3^k$, then
$\fm(\Z_3^k)=\left(\frac{(n-3)(n-1)}{3}+o(1)\right) 3^{n/9}.$  
\end{theorem}

Perhaps the next natural case to consider is 
$\fm (\mathbb Z^k_5)$. On the positive side, a recent result of Lev~\cite{lev} gives a structural characterisation of the large sum-free subsets of 
$\mathbb Z^k_5$; this result should very likely be needed in order to obtain asymptotically exact bounds on $\fm (\mathbb Z^k_5)$.
However, despite having this result at hand,
and receiving very nice suggestions from Vsevolod Lev and Wojciech Samotij, we  were unable to obtain such sharp bounds. 
We therefore state the following as a conjecture.
\begin{conjecture}\label{conj5}
    If $k \in \mathbb N$ and $n:=5^k$, then
$\fm (\mathbb Z^k_5)= 3^{n/10+o(n)}.$
\end{conjecture}

The following construction shows that the bound on  $\fm (\mathbb Z^k_5)$ in Conjecture~\ref{conj5} cannot be lowered. Let $B_2:=\{2\}\oplus\Z_5^{k-1}$, and $s:=\{1\}\oplus\{0\}$. The map $\phi:b\mapsto -b-s$ is an involution of $B_2$, and its only fixed point is $2s$. In particular, there are $(|B_2|-1)/2$ orbits $\{b,\phi(b)\}$ in $B_2$ of cardinality 2. Now consider the sets of the form
$$\{s\}\cup\bigcup_{\{b, \phi(b)\}\in B_2/\phi} A_b,$$
where the union ranges over all orbits of cardinality $2$ of $\phi$, and for each such orbit $\{b,\phi(b)\}$, we choose  $A_b$ to be one of $\{b,-b\}$, $\{b,-b-s\}$ or $\{b+s,-b-s\}$. This choice ensures that one can generate $3^{(|B_2|-1)/2}=3^{n/10-1/2}$  sets $S$ of this form. 
Further, each such set $S$ is sum-free; this follows as $S\setminus \{s\} \subseteq \{2,3\}\oplus\Z_5^{k-1}$ is sum-free and the only Schur triples (other than $s,s,2s$) in $\{s\} \cup (\{2,3\}\oplus\Z_5^{k-1})$ contain precisely two of $\{b,-b,b+s,-b-s\}$ for some $b \in B_2$.
Moreover, one 
can check that each of these sets lies in a different maximal sum-free subset  of $\Z_5^k$. Indeed, given two distinct such sets $S,S'$ the choice of $A_b$ for $S$ and $S'$ must differ for at least one $\{b, \phi(b)\}\in B_2/\phi$. But in this case $S \cup S'$ will contain a Schur triple (formed using $s$ and two of $b$, $-b$, $b+s$ and $-b-s$).
This proves that $\fm(\Z_5^k)\geq 3^{n/10-1/2}$.

\smallskip

Recently, further structural results for large sum-free subsets of 
$\mathbb Z^k_p$ were obtained when $p$ is a fixed prime where $p\equiv 2 \pmod{3}$; see~\cite{RS, ve}. These results may also be useful for the $\fm (\mathbb Z^k_p)$ problem more generally, although new ideas will again be needed to obtain sharp results.

\section{Maximal distinct sum-free subsets}
Let $G$ be an abelian group of order $n$, and let $f^\star(G)$ denote the number of distinct sum-free subsets in $G$. 
Let $f^\star_{\max} (G)$ denote the number of maximal distinct sum-free subsets in $G$. 
As pointed out in~\cite{HT}, the removal lemma of Green~\cite{G-R} implies that $\mu(G)\leq \mu ^\star(G) \leq \mu (G)+o(n)$ and $f(G) \leq f^\star(G) \leq 2^{o(n)} f(G)$. 
By arguing as in the proof that 
$f_{\max} (G) \leq 3^{\mu(G)/3+o(n)}$ for all abelian groups $G$ of order $n$~\cite[Proposition~5.1]{BLST2}, one can also conclude
that $f^\star_{\max} (G) \leq 3^{\mu(G)/3+o(n)}$.
However,
we do not know whether 
$f_{\max}(G)$ is bounded from above by the number of maximal distinct sum-free subsets $f_{\max}^\star(G)$ of $G$ for \emph{all} finite abelian groups $G$.
On the other hand, in~\cite[Section~6.3]{HT} we showed that
there are abelian groups $G$ for which $f_{\max}^\star(G)$ is exponentially larger than $f_{\max}(G)$.

At least for type \rom{1} abelian groups, we believe we have a clearer picture as to the behaviour of $f_{\max}^\star(G)$. Indeed, we conjecture the following.

\begin{conjecture}\label{conj2}
Suppose that  $G$ is a type $I$ abelian group of order $n$. Then $\fms(G)=2^{\mu (G)/2+o(n)}$.
\end{conjecture}
In~\cite{HT} we  established the $G=\mathbb Z_2^k$ case of Conjecture~\ref{conj2}. In fact, we showed that $\fm (\mathbb Z_2^k)=\fms(\mathbb Z_2^k)$, and so the result follows by Theorem~\ref{thm23}.

The next construction shows that, asymptotically, the bound on $f^\star_{\max}(G)$ in Conjecture~\ref{conj2} cannot be lowered in the case of type \rom{1}$(p)$ abelian groups for $p \geq 5$.
In Section~\ref{sec:mainproof}, we give  constructions that show the bound in  Conjecture~\ref{conj2} cannot be lowered in the case of even order groups ($p=2$); see Proposition~\ref{fms even groups2c}.

\begin{prop}\label{prop1}
Suppose $p\geq 5$ and $G$ is a type I(p) abelian group of order $n$. Then $\fms(G)\geq 2^{\mu (G)/2}$.
\end{prop}
\begin{proof}
    Let $B$ be a sum-free subset of $G$ (so in particular a distinct sum-free set) of size $\mu(G)=\frac{p+1}{3p}n$  and such that $B=-B$,\footnote{Actually, Lemma~\ref{large sfs groups I(p)} implies that every maximum size sum-free subset $B$ of $G$ satisfies $B=-B$.} e.g.,
    $$B:=\bigcup_{k=0}^{(p-2)/3}((3k+1)\cdot g+H),$$
    with $H$ a subgroup of $G$ of order $|G|/p$, and $g\not\in H$.
   The map $\sigma:x\mapsto -x$ is an involution of $B$, and since $p\geq5$, it has no fixed point. Then, there are $|B|/2$ orbits $\{x,-x\}$ of $\sigma$ with cardinality $2$. Choose one element $a_x$ in each  orbit $\{x, -x\}$, and consider the set
    $$\{0\}\cup\{a_x \ | \ \{x, -x\} \in B/\sigma\}.$$
    There are $2^{|B|/2}=2^{\mu(G)/2}$ ways to construct such a set, and notice each such set is distinct sum-free. Further, given two distinct sets $S,S'$ constructed in this way, $S$ and $S'$ must lie in different maximal distinct sum-free subsets of $G$. Indeed, $S \cup S'$  contains a Schur triple of the form $0, x,-x$. Thus, $\fms(G)\geq 2^{\mu (G)/2}$.
\end{proof}

Our next result  resolves Conjecture~\ref{conj2}   for most even order abelian groups $G$. 

\begin{thm}\label{fms even groups}
    There exists an absolute constant $C\in \mathbb N$ such that the following holds.
    Let $G$ be an  abelian group of even  order $n$, where $G=\Z_{2^{\alpha_1}}\oplus\ldots\oplus\Z_{2^{\alpha_r}}\oplus K,$
    with $r, \alpha_1,\dots, \alpha _r \in \mathbb N $, and $|K|$ odd. 
    If ${n}/{2^{r}}\geq C$, then 
    $$\fms(G)=2^{n/4+o(n)}.$$
\end{thm}
\begin{remarks} We now make a number of remarks concerning Theorem~\ref{fms even groups}.
\begin{itemize}
    
\item  Note that $r$ may depend on $n$ in the statement of Theorem~\ref{fms even groups}.

 \item  In Section~\ref{sec:mainproof}, we will see that one can certainly take $C:=10^{10}$.
 
\item  The condition that $n/2^{r} \geq C$  is analogous to a condition in~\cite[Theorem~3.2]{ls}, and is a very mild restriction. Indeed,  any even order abelian group $G$ not covered by Theorem~\ref{fms even groups} must be of the form 
$G=\mathbb Z_2 ^t \oplus H$ for some abelian group $H$ with $|H|<C^2$.

\item  Theorem~\ref{fms even groups} tells us that almost all  abelian groups $G$ of even order $n$ satisfy  $\fms(G)=2^{\mu(G)/2+o(n)}$. This is in contrast to the aforementioned result of Liu and Sharifzadeh~\cite[Theorem~3.2]{ls}, which demonstrates that almost all even order abelian groups $G$ satisfy $f_{\max}(G) \ll 2^{\mu (G)/2}$.

\end{itemize}
\end{remarks}

In Section~\ref{sec:mainproof} we prove -- via Theorems~\ref{fms even groups2} and~\ref{fms even groups2b} -- a sharp version of Theorem~\ref{fms even groups} that determines the term in front of $2^{n/4}$ rather than having the $o(n)$ term in the exponent.

Whilst Theorem~\ref{fms even groups}  makes substantial progress on the $\fm^\star(G)$ problem for even order abelian groups, the picture is still far from clear in general. As a next step,
it would be interesting to establish whether the analogue of Question~\ref{quesX} holds for $\fm^\star(G)$. In particular, note that by adapting an argument of  Liu and Sharifzadeh~\cite[Proposition 1.5]{ls}, one can deduce that 
$\fms(G)=3^{\mu(G)/3+o(n)}$
for all type \rom{2} groups $G$ whose order $n$ is divisible by $9$.

{\noindent \bf Additional note.} Since we submitted this paper, Balogh, Garcia, Liu and Yang~\cite{bgly}
have disproved Conjecture~\ref{conj2}. Indeed, they have shown that if 
$G:=\Z ^k_2 \oplus \Z_3$, then $\fms (G)=3^{(1+o(1)) \mu(G)/3}$. Further, for every 
prime $p \geq 23$ with $p\equiv 2 \pmod{3}$, they have shown that there are type \rom{1}$(p)$ groups $G$ with $\fms(G)$ exponentially larger than $2^{\mu(G)/2}$.  On the other hand, they have proven that Conjecture~\ref{conj2} is true for every even order group that is not isomorphic to $\Z ^k_2 \oplus \Z_3$; 
see~\cite[Theorem 1.7]{bgly}.
In light of this, it would be interesting to determine if Conjecture~\ref{conj2}
is true for `most' type $I$ abelian groups.

\smallskip

The rest of the paper is organised as follows. In the next section we introduce various  results and concepts that are useful for the proof of Theorem~\ref{fms even groups}, including the notion of a \emph{link graph}.
 In Section~\ref{sec:mainproof}  we prove (a strengthening of) Theorem~\ref{fms even groups} and provide  extremal examples for Conjecture~\ref{conj2} for all even order abelian groups.
In Section~\ref{sec:conj} we highlight a problem about maximal independent sets in graphs that we came across when considering Conjecture~\ref{conj5}.

\section{Useful results for Theorem~\ref{fms even groups}}\label{sec:usefulres}

\subsection{Tools for abelian groups} The following is a distinct sum-free subset analogue of a container result of Green and Rusza~\cite[Proposition~2.1]{GR-g}, and is central to our approach.

\begin{lemma}\label{distinct containers}
Given any $\delta>0$ there exists an $n_0 \in \mathbb N$ such that the following holds. If $G$ is a finite abelian group of order $n\geq n_0$, then there is a family $\mathcal{F}$ of subsets of $G$ with the following properties.
\begin{enumerate}[label = (\roman*),itemsep=0pt]
    \item Every $F\in\mathcal{F}$ has at most $\delta n^2$ distinct Schur triples.
    \item If $S\subseteq G$ is distinct sum-free, then $S$ is contained in some $F\in\mathcal{F}$.
    \item $|\mathcal{F}|\leq 2^{\delta n}$.
    \item Every $F\in \mathcal F$ is of the form $F= A \cup B$ where $A$ is sum-free and $|B|\leq \delta n$. In particular, $|F| \leq \mu (G)+\delta n$.
\end{enumerate} \qed
\end{lemma}
Note that Lemma~\ref{distinct containers} follows immediately from 
\cite[Theorem 1.5]{G-R} and, e.g., \cite[Theorem~2.2]{bms}.
We refer to the sets in $\mathcal{F}$ as \textit{containers}. 
With Lemma~\ref{distinct containers} at hand, we see that in order to count the number of maximal distinct sum-free sets in an abelian group $G$, it suffices to count the number of maximal distinct sum-free sets in each container. To do so, the next  result provides structure on the sum-free sets in type I$(p)$ groups.

\begin{lemma}\cite[Lemma 5.6]{GR-g}\label{large sfs groups I(p)}
    Suppose that $G$ is a type I($p$) group  and write $p=3k+2$. Let $A\subseteq G$ be sum-free, and suppose that $|A|>\left(\frac{1}{3}+\frac{1}{3(p+1)}\right)n$. Then we may find a homomorphism $\psi:G\to\Z_p$ such that $A$ is contained in $\psi^{-1}(\{k+1,\ldots,2k+1\})$.
\end{lemma}

In the proof of Theorem~\ref{fms even groups}, 
we will use  the following easy fact about the number of solutions of the equation $2x=g$ in an even order abelian group.

\begin{fact}\label{solutions 2x=g}
    Let 
$G:=\Z_{2^{\alpha_1}}\oplus\Z_{2^{\alpha_2}}\oplus\ldots\oplus\Z_{2^{\alpha_r}}\oplus K $ where $K$ is an abelian group of odd order and $\alpha _i \in \mathbb N$ for each $ i \in [r]$. Fix $g\in G$. Then there are at most $2^r$ solutions in $G$ to the equation $2x=g$.\qed
\end{fact}

We end this subsection with a technical lemma that will help us make a sharp count on the number of maximal distinct sum-free subsets in even order abelian groups with few  elements of order $2$. To prove it, we will use the following well-known correspondence (e.g., it is an immediate corollary of the first proposition in~\cite[Section 2.2]{mallet-burgues}).

\begin{lemma}\label{corr order index}
    Let $G$ be a finite abelian group and $k\in\N$. The  number of subgroups of $G$ of order $k$ equals the number of subgroups of $G$ with index $k$. \qed
\end{lemma}

We also require the following fact about subgroups of maximal rank in some even order abelian groups.

\begin{fact}\label{nb order 2 elements}
     Given $\alpha_1\geq\ldots\geq\alpha_r\geq1$, let $G:=\Z_{2^{\alpha_1}}\oplus\ldots\oplus\Z_{2^{\alpha_r}}$ be a group of rank $r$. If $H\leq G$ is a subgroup of $G$, then $H$ has rank $r$ if and only if it contains every order 2 element of $G$.
\end{fact}
\begin{proof}
    First note that the order 2 elements in $G$ are the non-zero elements of the set $\{0,2^{\alpha_1-1}\}\oplus\ldots\oplus\{0,2^{\alpha_r-1}\}$, so there are $2^r-1$ of them. Since $H\leq G$, there exist $\beta_1\geq\ldots\geq\beta_{r'}\geq1$ such that $H\cong \Z_{2^{\beta_1}}\oplus\ldots\oplus\Z_{2^{\beta_{r'}}}$, with $r'\leq r$ the rank of $H$. Thus, $H$ has $2^{r'}-1$ order 2 elements. So $H$ contains every order 2 element of $G$ if and only if $2^r-1=2^{r'}-1$, i.e., if and only if $r=r'$.
\end{proof}

The final lemma of this subsection counts the number of index 2 subgroups of a given rank in some even order abelian groups. Recall that given an abelian group $G$,   $2G$ denotes the subgroup of $G$ whose elements are of the form $2g$ for some $g \in G$.

\begin{lemma}\label{nb subgroups of index 2 with fixed rank}
Let $r_1$ and $r_2$ be non-negative integers such that $r:=r_1+r_2\in \mathbb N.$ 
    Let $G:=\Z_{2^{\alpha_1}}\oplus\ldots\oplus\Z_{2^{\alpha_{r_1}}}\oplus\Z_2^{r_2}$ with $\alpha_1\geq\ldots\geq\alpha_{r_1}\geq2$. Then there are $2^{r_1}-1$ subgroups of index 2 of $G$ with rank $r$, and $2^r-2^{r_1}$ subgroups of index 2 of $G$ with rank $r-1$. 
\end{lemma}

\begin{proof}
    Note that a subgroup of index 2 of $G$ has either rank $r$ or $r-1$. Let $\mathcal{I}_2(G)$ be the set of all the subgroups of $G$ with index 2. Let $\mathcal{R}(G)$ (resp.~$\mathcal{R}_{-1}(G)$) be the set of the subgroups of index 2 of $G$ with rank $r$ (resp.~$r-1$). Set $\rho(G):=|\mathcal{R}(G)|$ and $\rho_{-1}(G):=|\mathcal{R}_{-1}(G)|$. By Lemma \ref{corr order index}, $|\mathcal{I}_2(G)|$ is equal to the number of subgroups of order 2 of $G$ (which is the number of order 2 elements in $G$), namely $2^r-1$.
    Then, $\rho(G)+\rho_{-1}(G)=2^r-1$.

    We will now define a bijection $\varphi$ between $\mathcal{R}(G)$ and $\mathcal{I}_2(2G)$. 
    
    Let $H\in\mathcal{R}(G)$, and define $\varphi(H):=2H$. Let us prove that $\varphi(H)\in\mathcal{I}_2(2G)$. Consider the homomorphism $\pi:G\to 2G/2H$ defined by $\pi(g)=2g+2H$. Then $g\in\ker(\pi)$ if and only if $2g \in 2H$, which is true precisely if  (i) $g \in H$ or (ii) there exists $h\in H$ such that $g-h$ has order 2. Note though that in case (ii), $g-h \in H$ since $H$ has rank $r$ and so by Fact \ref{nb order 2 elements} it contains every element of $G$ of order 2; this then implies $g \in H$. So in fact $g\in\ker(\pi)$ if and only if $g \in H$ and thus
     $\ker(\pi)=H$.
By the fundamental theorem on homomorphisms, this implies that 
    $2G/2H$ is isomorphic to $G/H$, proving that $2H\in\mathcal{I}_2(2G)$. This proves that $\varphi$ is well-defined.

    Let $H,H'\in\mathcal{R}(G)$, and suppose that $\varphi(H)=\varphi(H')$, i.e., $2H=2H'$. Then for every $h\in H$, either $h \in H'$ or there exists $h'\in H'$ such that $h-h'$ has order 2. But since $H'\in\mathcal{R}(G)$, it contains every order 2 element of $G$ by Fact \ref{nb order 2 elements}; this means $h-h'\in H'$, so $h\in H'$. Thus, $H\subseteq H'$. Similarly $H'\subseteq H$, so $H=H'$, and $\varphi$ is injective.
    
    Conversely, let $\Tilde{H}\in\mathcal{I}_2(2G)$. Then $H:=\{g\in G \ | \ 2g\in\Tilde{H}\}$ is such that $\varphi(H)=\Tilde{H}$. Indeed, consider the homomorphism 
    $\Tilde{\pi}:G\to 2G/\Tilde{H}$ defined by  $\Tilde{\pi}(g)=2g+\Tilde{H}$. Then by definition $H=\ker(\Tilde{\pi})$, 
    so $H$ is an index 2 subgroup of $G$.
    Further, $H$ contains every  element in $G$ of order $2$, and so by Fact \ref{nb order 2 elements} must have rank $r$; that is, $H\in\mathcal{R}(G)$. This therefore shows $\varphi$ is surjective.

    This bijection $\varphi$  proves that $\rho(G)=|\mathcal{I}_2(2G)|$. Note that $2G$ has rank $r_1$, and so since $2G$ contains  $2^{r_1}-1$ order 2 elements, Lemma~\ref{corr order index} implies that  $|\mathcal{I}_2(2G)|=2^{r_1}-1$.
    Thus, $\rho(G)=2^{r_1}-1$, and since $\rho(G)+\rho_{-1}(G)=2^r-1$ we also have that $\rho_{-1}(G)=2^r-2^{r_1}$.
\end{proof}

\subsection{Maximal independent sets in graphs}

A key aspect of the approach we take is the translation of the problem of maximal sum-free sets in abelian groups into one about maximal independent sets in graphs. Therefore, we now introduce  some notation, and  state several results that bound the number of maximal independent sets in graphs.

Let $\Gamma=(V,E)$ be a graph. 
We define $v(\Gamma):=|V|$ and $e(\Gamma):=|E|$.
Given a vertex $x\in V$, we write $d(x,\Gamma)$ for the \textit{degree} of $x$ in $\Gamma$ (i.e., the number of edges in $G$ incident to $x$). We write $\delta(\Gamma)$ for the \textit{minimum degree} and $\Delta(\Gamma)$ for the \textit{maximum degree} of $\Gamma$. 
Given a subset $X \subseteq V$ we write $\Gamma \setminus X$ for the subgraph of $\Gamma$ induced by $V \setminus X$.

We write $K_m$ for the complete graph on $m$ vertices and $C_m$  for the cycle on $m$ vertices. Given graphs $\Gamma$ and $\Gamma'$ we write $\Gamma\square\Gamma'$ for the \textit{cartesian product graph}; so its vertex set is $V(\Gamma)\times V(\Gamma')$ and $(x,y)$ and $(x',y')$ are adjacent in $\Gamma\square\Gamma'$ if (i) $x=x'$ and $y$ and $y'$ are adjacent in $\Gamma'$ or (ii) $y=y'$ and $x$ and $x'$ are adjacent in $\Gamma$. 

We denote by $\mis(\Gamma)$ the number of maximal independent sets in $\Gamma$. Moon and Moser \cite{MM} proved the following bound which holds for any $n$-vertex graph $\Gamma$:
\begin{equation}\label{Moon-Moser}
\mis(\Gamma)\leq 3^{n/3}.
\end{equation}
Note the bound in (\ref{Moon-Moser}) is tight; consider a graph consisting of the disjoint union of triangles. However, Hujter and Tuza~\cite{HT2} improved this bound in the case of  triangle-free $n$-vertex graphs $\Gamma$:
\begin{equation}\label{Hujter-Tuza}
\mis(\Gamma)\leq 2^{n/2}.
\end{equation}
If $\Gamma$ is a perfect matching then we have equality in (\ref{Hujter-Tuza}). The following lemma improves on (\ref{Moon-Moser}) in the case of somewhat regular and dense graphs.

\begin{lemma}\cite[Equation (3)]{BLST}\label{dense and regular graphs}
Let $k\geq 1$ and let $\Gamma$ be a graph on $n$ vertices. Suppose that $\Delta(\Gamma)\leq k\delta(\Gamma)$ and set $b:=\sqrt{\delta(\Gamma)}$. Then
$$\mis(\Gamma)\leq \sum_{0\leq i\leq n/b}\binom{n}{i}\cdot3^{\left(\frac{k}{k+1}\right)\frac{n}{3}+\frac{2n}{3b}}.$$
\end{lemma}

We will also use the following  refined versions of (\ref{Hujter-Tuza}) and (\ref{Moon-Moser}), respectively.

\begin{lemma}\cite[Corollary 3.3]{BLST2}\label{almosttrianglefree}
    Let $n,D\in\N$ and $k\in\R$. Suppose that $\Gamma$ is a graph and $T$ is a set such that $\Gamma':=\Gamma\setminus T$ is triangle-free. Suppose that $\Delta(\Gamma)\leq D$, $v(\Gamma')=n$ and $e(\Gamma')\geq n/2+k$. Then 
    $$\mis(\Gamma)\leq 2^{n/2-k/(100D^2)+2|T|}.$$
\end{lemma}

\begin{lemma}\cite[Lemma 3.5]{ls}\label{stability}
Let $k\in\Z$, $\Delta\in\N$, and $C:= 3^{\Delta/13}$. If $\Gamma$ is an $n$-vertex graph with $n+k$ edges and maximum degree $\Delta$, then
$$\mis(\Gamma)\leq C\cdot 3^{\frac{n}{3}-\frac{k}{13\Delta}}.$$
\end{lemma}

The next definition provides the graph that connects our distinct sum-free sets problem to independent sets in graphs.

\begin{definition}
    For subsets $B,S\subseteq G$, let $L^{\star}_S[B]$ be the \emph{distinct link graph of $S$ on $B$} defined as follows. Its vertex set is $B$ and its edge set consists of the following edges:

\begin{enumerate}[label = (\roman*),itemsep=0pt]
    \item two distinct vertices $x,y\in B$ are adjacent if there exists $s\in S$ such that $\{x,y,s\}$ is a distinct Schur triple;
    \item there is a loop at a vertex $x\in B$ if there exist distinct $s,s'\in S$  such that $\{x,s,s'\}$ is a distinct Schur triple.
\end{enumerate}
We call an edge $xy$ in $L^{\star}_S[B]$ a \emph{type 1} edge if $x-y=s$ for some $s\in S\cup (-S)$, and a \emph{type 2} edge if it is not a type 1 edge and $x+y=s$ for some $s\in S$. We denote by $d_i(x,L^{\star}_S[B])$ the number of type $i$ edges incident to $x$ in $L^{\star}_S[B]$. We write $e_i(L^{\star}_S[B])$ for the number of types $i$ edges in $L^{\star}_S[B]$.
\end{definition}

\begin{figure}[h]
    \centering
    \begin{tikzpicture}[scale=1.3]
\draw[blue] (-0.5,0) -- ++(-1.5,0.866) -- ++(0,-1.732) -- cycle;
\draw[blue] (0.5,0) -- ++(1.5,0.866) -- ++(0,-1.732) -- cycle;
\draw[red] (-0.5,0) -- (0.5,0);
\draw[red] (-2,0.866) -- (2,0.866);
\draw[red] (-2,-0.866) -- (2,-0.866);
\node at (-0.3,0.3) {$x+s$};
\node at (0.45,-0.3) {$-x$};
\node at (-2.35,0.866) {$x$};
\node at (-2.65,-0.866) {$x+2s$};
\node at (2.75,0.866) {$-x-2s$};
\node at (2.7,-0.866) {$-x-s$};
\fill[black] (0.5,0) circle (2pt);
\fill[black] (-0.5,0) circle (2pt);
\fill[black] (-2,0.866) circle (2pt);
\fill[black] (-2,-0.866) circle (2pt);
\fill[black] (2,0.866) circle (2pt);
\fill[black] (2,-0.866) circle (2pt);
\end{tikzpicture}
    \caption{\centering Let $S=\{s\}$ such that $s$ has order 3. If $x\in B$ is such that $2x\not\in\{0,s,2s\}$, then $s$ creates 9 edges in the component containing $x$ in $L^{\star}_S[B]$, as shown above. Type 1 edges are represented in blue, and type 2 edges in red.}
    \label{fig:ex distinct link graph 1}
\end{figure}

\begin{figure}[h]
    \centering
\begin{tikzpicture}[scale=1.3]
\draw[blue] (0,0) -- (2,0);
\draw[red] (0,0) -- (-2,0);
\draw[red] (-2,0) arc (-90:-480:0.25cm);
\fill[black] (0,0) circle (2pt);
\fill[black] (-2,0) circle (2pt);
\fill[black] (2,0) circle (2pt);
\node at (0,-0.3) {$2s'$};
\node at (-2,-0.3) {$s+s'$};
\node at (2,-0.3) {$2s+2s'$};
\end{tikzpicture}
    \caption{\centering The distinct link graph $L^{\star}_{\{s,s'\}}[B]$, where $s,s'$ have order 3 and $B=\{s+s',2s',2s+2s'\}$. Note that $B$ is distinct sum-free but not sum-free.}
    \label{fig:ex distinct link graph 2}
\end{figure}
In Figures~1 and~2 we give a couple of examples of distinct link graphs.
The following lemma provides the connection between maximal independent sets in the distinct link graph and maximal distinct sum-free sets.

\begin{lemma}\label{extension link graph}
Suppose that $B,S\subseteq G$ are both distinct sum-free. If $I\subseteq B$ is such that $S\cup I$ is a maximal distinct sum-free subset of $G$, then $I$ is a maximal independent set in $L_S^\star[B]$.   \qed 
\end{lemma}
 Note that the proof of Lemma~\ref{extension link graph} is identical to the proof \cite[Lemma 3.1]{BLST} (which deals with  maximal sum-free subsets of $[n]$).

\section{Proof of Theorem~\ref{fms even groups}}\label{sec:mainproof}

In this section we prove Theorem \ref{fms even groups}. The proof is based on the general approach initiated in~\cite{BLST, BLST2}, and also the proof adapts parts of arguments that appear in~\cite{HT, ls}.
Note that  Theorem~\ref{fms even groups} follows immediately from the following two results.

\begin{thm}[Sharp upper bound for even order abelian groups]\label{fms even groups2}
    There exists an absolute constant $C\in \mathbb N$ such that the following holds. 
    Given any  $\eta >0$, there exists an $n_0 \in \mathbb N$ such that the following holds for all even $n \geq n_0$.
    Let $G$ be an  abelian group of   order $n$, where $G=\Z_{2^{\alpha_1}}\oplus\ldots\oplus\Z_{2^{\alpha_{r_1}}}\oplus\Z_2^{r_2}\oplus K,$
    with $r_1$ and $r_2$ non-negative integers, $ \alpha_1\geq\ldots\geq\alpha_{r_1}\geq2$, and $|K|$ odd. Set $r:=r_1+r_2\in \mathbb N$. If ${n}/{2^r}\geq C$, then 
    $$ \fms(G)\leq (2^{2r-1}+2^{r+r_1-1}-2^{r+1}+2^{r_1})2^{n/4}+(2^r-2^{r_1}+\eta)2^{n/4-2^{r-2}}.$$
\end{thm}
The following theorem shows that the upper bound on $\fms(G)$ in Theorem~\ref{fms even groups2} is essentially best possible.
\begin{thm}[Sharp lower bound for even order abelian groups]\label{fms even groups2b}
    There exists an absolute constant $C\in \mathbb N$ such that the following holds. Given any  $\eta >0$, there exists an $n_0 \in \mathbb N$ such that the following holds for all even $n \geq n_0$.
    Let $G$ be an  abelian group of   order $n$, where $G=\Z_{2^{\alpha_1}}\oplus\ldots\oplus\Z_{2^{\alpha_{r_1}}}\oplus\Z_2^{r_2}\oplus K,$
    with $r_1$ and $r_2$ non-negative integers, $ \alpha_1\geq\ldots\geq\alpha_{r_1}\geq2$, and $|K|$ odd. Set $r:=r_1+r_2\in \mathbb N$. If ${n}/{2^r}\geq C$, then 
    $$ (2^{2r-1}+2^{r+r_1-1}-2^{r+1}+2^{r_1})2^{n/4}+(2^r-2^{r_1}-\eta)2^{n/4-2^{r-2}} \leq \fms(G).$$
\end{thm}
Together Theorems~\ref{fms even groups2} and~\ref{fms even groups2b} provide a sharp version of  Theorem~\ref{fms even groups}. Though the case $G=\mathbb Z_2^k$ is not covered by these two theorems, notice that the coefficient in front of $2^{n/4}$ in 
Theorems~\ref{fms even groups2} and~\ref{fms even groups2b} agrees with coefficient in front of  $2^{n/4}$ in Theorem~\ref{thm23} (recall that $\fms (\mathbb Z_2^k)=\fm(\mathbb Z_2^k)$).
 
The next result  provides an extremal example for Conjecture~\ref{conj2} for each even order abelian group.

\begin{prop}[General lower bound for even order abelian groups]\label{fms even groups2c}
    Suppose $G$ is an  abelian group of   even order $n$.
    Then 
    $\fms(G)\geq 2^{(n-2)/4}.$
\end{prop}
\begin{proof}
 Let $G$ be an  abelian group of even  order $n$.
 Thus,
$G=\Z_{2^{\alpha_1}}\oplus\ldots\oplus\Z_{2^{\alpha_{r_1}}}\oplus\Z_2^{r_2}\oplus K$
    for some  non-negative integers $r_1$ and $r_2$, such that $r:=r_1+r_2\in \mathbb N$, $\alpha_1\geq\ldots\geq\alpha_{r_1}\geq2$, and $|K|$ is odd.
    Let $H$ be an index 2 subgroup of 
$\Z_{2^{\alpha_1}}\oplus\ldots\oplus\Z_{2^{\alpha_{r_1}}}\oplus\Z_2^{r_2}$ and let
   $\Tilde{A}$ be the coset associated to $H$ that does not contain the identity (i.e., the coset such that $\Tilde{A}\ne H$). 

   We split our argument into  two cases. First, suppose that $r_1 \geq 1$ or $r_2 \geq 2$. In this case $|H|$ is even, so there is some $s\in G\setminus (\Tilde{A} \oplus K)=H \oplus K$ of order 2.
    Set $\Gamma:=L_{\{0,s\}}^\star[\Tilde{A}\oplus K]$. We now have two subcases depending on the rank of $H$.
    \begin{enumerate}[label=(\roman*)]
        \item $H$ has rank $r-1$. In this subcase,
         if $x \in \Tilde{A}\oplus K$ is such that $2x\not \in\{0,s\}$, then the component of $\Gamma $ containing $x$ is precisely a 
         $K_4$ containing the vertices  $x,x+s,s-x,-x$. If $2x=0$, then $x$'s component in $\Gamma$ is the  edge $\{x,x+s\}$; if 
         $2x=s$, then $x$'s component in $\Gamma$ is the edge $\{x,-x\}$. 
         Since $H$ has rank $r-1$, there are $2^{r-1}$ order 2 elements in $\Tilde{A} \oplus K$. Let us denote by $a(s)$ the number of 
         elements $x\in \Tilde{A} \oplus K$ such that $2x=s$. Since $\mis(K_4)=4$ and $\mis(K_2)=2$, we then have
        $$\mis(\Gamma)=2^{\frac{2^{r-1}+a(s)}{2}}\cdot 4^{\frac{n/2-2^{r-1}-a(s)}{4}}=2^{n/4}.$$
        \item $H$ has rank $r$. In this subcase there are no order 2 elements in $\Tilde{A} \oplus K$, so $\Gamma$ is the disjoint union of $a(s)/2$ copies of $K_2$ and $(n/2-a(s))/4$ copies of $K_4$. Thus,
        $$\mis(\Gamma)=2^{\frac{a(s)}{2}}\cdot 4^{\frac{n/2-a(s)}{4}}=2^{n/4}.$$
    \end{enumerate}
    In both subcases, $\mis(\Gamma)=2^{n/4}$. Now if $I$ is a maximal independent set in $\Gamma$, then $\{0,s\}\cup I$ is a distinct sum-free subset of $G$, however, it might not be maximal. But if $I'\ne I$ is another maximal independent set in $\Gamma$, then $\{0,s\}\cup I$ and $\{0,s\}\cup I'$ must lie in different maximal distinct sum-free subsets of $G$; indeed, otherwise  $I\cup I'$ is an independent set in $\Gamma$, a contradiction as $I$ and $I'$ are maximal independent sets in $\Gamma$. Hence, $\fms(G)\geq\mis(\Gamma)=2^{n/4}$ in this case.

    Finally, suppose that $r_1=0$ and $r_2=1$. Thus, $G=\Z_2\oplus K$. Set $\Gamma:=L_{\{0\}}^\star[\{1\}\oplus K]$.
    In this case, $\Gamma$ consists of
    an isolated vertex $\{1\}\oplus\{0\}$, together with 
    a matching where edges are of the form $\{x, -x\}$ for $x \in \{1\}\oplus K$.  Thus, $\mis(\Gamma)= 2^{\frac{|K|-1}{2}}= 2^{(n-2)/4}$. Arguing as in the first case gives 
    $\fms(G)\geq\mis(\Gamma)= 2^{(n-2)/4}$, as desired.
\end{proof}
Note that one can rephrase the proof of Proposition~\ref{prop1} in terms of link graphs as in the proof of Proposition~\ref{fms even groups2c}.

\subsection{Proof of Theorem~\ref{fms even groups2}}

    Let $C:=3235084117$; in footnote 2 we point out the  place in the argument where this precise choice of $C$ is used. Given any $\eta >0$, define additional constants $\delta, \eps, n_0 >0$ so that
$$0< 1/n_0 \ll \delta \ll \eps \ll \eta, 1/C.$$
Here the additional constants in the hierarchy are chosen from right to left.

Let $G$ be an abelian group of even order $n \geq n_0$ as in the statement of the theorem. 
Thus,
$G=\Z_{2^{\alpha_1}}\oplus\ldots\oplus\Z_{2^{\alpha_{r_1}}}\oplus\Z_2^{r_2}\oplus K$ with $r_1$ and $r_2$ non-negative integers, $\alpha_1\geq\ldots\geq\alpha_{r_1}\geq2$,  and $|K|$ odd. Furthermore, $n/2^{r_1+r_2}\geq C$.

Apply Lemma~\ref{distinct containers} to $G$ with parameter $\delta $ to obtain a collection $\mathcal F$ of containers satisfying conditions (i)--(iv) of the lemma. 

 By Lemma~\ref{distinct containers}(ii), to obtain an upper bound on $\fms(G)$ it suffices to bound the number of maximal distinct sum-free subsets of $G$ that lie in the containers.
Given any $F \in \mathcal F $,
let $\fms(F)$ denote the number of maximal distinct sum-free subsets of $G$ that lie in $F$.
Lemma~\ref{distinct containers}(iv) implies that
$F=A\cup B$ where $A$ and $B$ are disjoint, $A$ is sum-free and $|B|\leq \delta n$.
Thus, each maximal distinct sum-free subset of $G$ lying in $F$ can be obtained in the following way:
\begin{enumerate}[label=(\arabic*), itemsep=0pt]
    \item Choose a (perhaps empty) distinct sum-free set $S$ in $B$;
    \item Extend $S$ in $A$ to a maximal one.
\end{enumerate}

This two-step approach, combined with Lemma~\ref{extension link graph}, shows that 
$$\fms(F)\leq 2^{\delta n}\max_{\substack{S\subseteq B\\ S \text{ is distinct sum-free}}} \mis(L_S^\star[A]),$$
where $L_S^\star[A]$ is the distinct link graph of $S$ on $A$. 
With this in mind, we now define 7 different types of maximal distinct sum-free subsets of $G$.
\begin{enumerate}
    \item[\textbf{Type 0:}] those obtained from a container $F$ where $S$ is chosen to be empty.
    \item[\textbf{Type 1:}] those obtained from a container $F$ where $|A|\leq 4n/9$ and $|S|\geq 1$.
    \item[\textbf{Type 2:}] those obtained from a container $F$ where $|A|> 4n/9$ and  $|S|> 87^2$.
    \item[\textbf{Type 3:}] those obtained from a container $F$ where $|A|> 4n/9$ and $2\leq|S|\leq 87^2$, with $S\ne\{0,s\}$ for every order 2 element $s\in G$.
    \item[\textbf{Type 4:}] those obtained from a container $F$ where $|A|> 4n/9$ and $S=\{s\}$ with $s\ne0$.
    \item[\textbf{Type 5:}] those obtained from a container $F$ where $|A|> 4n/9$ and $S=\{0\}$.
    \item[\textbf{Type 6:}] those obtained from a container $F$ where $|A|> 4n/9$ and $S=\{0,s\}$ with $s$ an order 2 element.
\end{enumerate}

Let $f_{\max}^{\star,i}(G)$ be the total number of maximal distinct sum-free sets of type $i$ in $G$. Similarly, we write $f_{\max}^{\star,i}(F)$ for the total number of maximal distinct sum-free sets that lie in the container $F$. Note that each container $F$ can produce at most one type 0 maximal distinct sum-free set (namely $A$), so $f_{\max}^{\star,0}(G)\leq |\mathcal F| \leq 2^{\delta n}$ by Lemma \ref{distinct containers}.

Consider any container $F\in \mathcal F$.
Now we fix a non-empty distinct sum-free set $S\subseteq B$. In what follows we count how many ways we can extend $S$ in $A$ to a maximal distinct sum-free subset of $G$. Let $\Gamma:=L_S^\star[A]$ be the distinct link graph of $S$ on $A$. From Lemma \ref{extension link graph}, we see that the number of extensions of $S$ in $B$ to a maximal sum-free set is at most $\mis(\Gamma)$.

The next claim will be used to bound $f_{\max}^{\star,1}(G)$.
\begin{claim}\label{type 1}
    If $|A|\leq 4n/9$, then $\mis(\Gamma)\leq 2^{(1/4-\eps)n-2^{r-2}}$.
\end{claim}
\begin{proof}
     By the Moon--Moser bound (\ref{Moon-Moser}),
    $$\mis(\Gamma)\leq 3^{|A|/3}\leq 3^{4n/27}.$$
    The choice of $\eps$ and $C$ implies that $3^{4n/27}\leq2^{(1/4-\eps)n-2^{r-2}}$ since $n/2^r\geq C$.
\end{proof}

Now we consider type 2--6 maximal distinct sum-free sets. Suppose $|A|>4n/9$;  by Lemma~\ref{large sfs groups I(p)} we can assume that $A\subseteq\Tilde{A}\oplus K$ with $\Tilde{A}$ a coset (not containing the zero element) associated to a subgroup of index 2 in $\Z_{2^{\alpha_1}}\oplus\ldots\oplus\Z_{2^{\alpha_{r_1}}}\oplus\Z_2^{r_2}$. 
In fact, (by adding elements to the container $F$ if necessary) we may assume $A=\Tilde{A}\oplus K$, as $f_{\max}^{\star,i}(F)\leq f_{\max}^{\star,i}(F')$ for $F'\supseteq F$, $i\in [6]$.  So $A$ is still sum-free.
We can also  still assume that $A$ and $B \supseteq S$  are disjoint.
Under these conditions, we obtain the following information about the regularity of $\Gamma$.

\begin{claim}\label{degree distinct link graph} Let $x\in A$. Then,
    \begin{itemize}
        \item $d_1(x,\Gamma)=|(S\cup-S)\setminus \{0\}|$;
        \item $d_2(x,\Gamma)\leq |S|$;
        \item $\Gamma $ does not contain any loops.
    \end{itemize}
    In particular, 
    $$|(S\cup-S)\setminus \{0\}|\leq \delta(\Gamma)\leq \Delta(\Gamma)\leq |(S\cup-S)\setminus \{0\}| + |S|.$$
    Consequently,
    \begin{enumerate}
        \item If $0\in S$ and $S=-S$, then $\Delta(\Gamma)\leq 2\delta(\Gamma)+1$.
        \item Otherwise, $\Delta(\Gamma)\leq 2\delta(\Gamma)$.
    \end{enumerate}
\end{claim}

\begin{proof}
    For the first part, note that each element $s$ of $(S\cup-S)\setminus \{0\}$ generates a unique type 1 neighbour $y=x+s$. For the second part, we have that $x$ is incident to all edges of the form $(x,s-x)$, with $s\in S$ such that $x\ne s-x$; we have an upper bound because, e.g.,  some of those edges might be type 1. 

    Note that our assumptions ensure $A=g + H$ where $H$ is some index $2$ subgroup of $G$ and $g \notin H$, and where $S \subseteq H$. This, together with the definition of the distinct link graph, ensures $\Gamma$ does not contain loops.
    
    The only possibility for $|(S\cup-S)\setminus \{0\}|+|S|$ to be greater than $2|(S\cup-S)\setminus \{0\}|$ is that $0\in S$, and $S=-S$. Hence, (1) and (2) follow.
\end{proof}
\begin{claim}\label{nb edges distinct link graph}
    $e(\Gamma)\geq \frac{|(S\cup -S)\setminus \{0\}|+|S|}{2}|A|-|S|(|S\cup -S|+1)2^r\geq \frac{\Delta(\Gamma)|A|}{2}-|S|(|S\cup -S|+1)2^r$.
\end{claim}
\begin{proof}
    By Claim~\ref{degree distinct link graph}, we only need to prove that $e_2(\Gamma)\geq \frac{|S||A|}{2}-|S|(|S\cup -S|+1)2^r$.
    
Let $X$ denote the number of  pairs $(x,s-x)$ where $x \in A$ and $s \in S$ such that either (i) $(x,s-x)$ is a type 1 edge in $\Gamma$ or (ii) $x=s-x$. Note that $e_2 (\Gamma) \geq  \frac{|S||A|}{2}-X$.

   Let $x\in A$ and $s\in S$. Then $(x,s-x)$ is a type 1 edge in $\Gamma$ if and only if there exists $s'\in S\cup -S$ such that $x-(s-x)=s'$, or equivalently, $2x=s+s'$. By Fact~\ref{solutions 2x=g}, we deduce there are at most $|S||S\cup -S|2^r$ such edges in $\Gamma$.

Similarly, given any $s \in S$, Fact~\ref{solutions 2x=g} implies that there are at most $2^r$ solutions $x \in A$ to $x=s-x$. Together with the last paragraph, this implies that $X \leq |S||S\cup -S|2^r +|S|2^r$, as required.
\end{proof}
The next claim will be used to bound
$f_{\max}^{\star,2}(G)$.
\begin{claim}\label{type 2}
    For any $S\subseteq B$ with $|S|> 87^2$, we have $\mis(\Gamma)\leq2^{(1/4-\eps)n-2^{r-2}}$.
\end{claim}
\begin{proof}
     By Claim \ref{degree distinct link graph}, $\Delta(\Gamma)\leq 3\delta(\Gamma)$. We apply Lemma \ref{dense and regular graphs} with $k=3$. Noting that $v(\Gamma)=n/2$ and $b:=\sqrt{\delta(\Gamma)}\geq 87$, we thus obtain that
    $$\mis(\Gamma)\leq\sum_{0\leq i\leq n/174}\binom{n/2}{i}3^{\frac{n}{8}+\frac{n}{261}}\leq \left(\frac{n}{174}+1\right)(87e)^{\frac{n}{174}}3^{\frac{n}{8}+\frac{n}{261}}.$$
    The choice of $C$ and $\eps$, together with the fact that $n/2^r\geq C$ and $n$ is sufficiently large, ensures that $\left(\frac{n}{174}+1\right)(87e)^{\frac{n}{174}}3^{\frac{n}{8}+\frac{n}{261}}\leq 2^{(1/4-\eps)n-2^{r-2}}$, as required.
\end{proof}
The next claim will be used to bound
$f_{\max}^{\star,3}(G)$.

\begin{claim}\label{type 3}
    Consider any distinct sum-free subset $S\subseteq B$ with $2\leq|S|\leq 87^2$ and $S\ne\{0,s\}$ for every order 2 element $s\in G$.
    Then $\mis(\Gamma)\leq2^{(1/4-\eps)n-2^{r-2}}$.
\end{claim}
\begin{proof}
    Our first aim is to show that $\Delta (\Gamma) \geq 4$. We will prove this by considering  8 sub-cases:
\begin{enumerate}[label=(\roman*)]
    \item $|S|=4$ and 
    $S=\{0,s_1,s_2,s_3\}$ with $s_1,s_2,s_3$ having order 2.

    \item $|S|=3$ and 
    $S=\{s_1,s_2,s_3\}$ with $s_1,s_2,s_3$ having order 2.
    \item $|S|=3$ and 
    $S=\{s_1,s_2,-s_2\}$ with $s_1$ having order 2.
    \item $|S|=3$ and 
    $S=\{0,s_1,s_2\}$ with $s_1$ having order 2.
    \item $|S|=2$ and 
    $S=\{s,-s\}$.
    \item $|S|=2$ and 
    $S=\{s_1,s_2\}$ with $s_1$ having order 2 and $s_2 \neq 0$.
    \item $|S|=2$ and 
    $S=\{0,s\}$ with $s\ne -s$.
    \item Any other $S$ satisfying the hypothesis of the claim.
\end{enumerate} 
Set $\Delta:= \Delta (\Gamma)$.
One can check that in case (viii),  Claim~\ref{degree distinct link graph} implies that $\Delta\geq |(S\cup -S)\setminus \{0\}|\geq 4$.
For example, if $|S|=4$ and $S$ does not satisfy (i), then  $0 \notin S$ or $-S\setminus S \neq \emptyset$   (note that here we are using that $S$ is distinct sum-free). Then Claim~\ref{degree distinct link graph} ensures that $\Delta\geq 4$.

We can also check by hand that $\Delta\geq4$ in cases (i)--(vii). Indeed, in cases (i) and (ii) consider a vertex $x \in V(\Gamma)$ such that  $2x \not\in \{s_1,s_1-s_2,s_1-s_3,0\}$; note that such an $x \in V(\Gamma)$ exists by Fact~\ref{solutions 2x=g}
and as $|\Gamma|=n/2 >4 \cdot 2^r$. Then $x$ has 
 degree at least $4$ in $\Gamma$ since $x+s_1$, $x+s_2$, $x+s_3$ and $s_1-x$ are distinct neighbours of $x$ in $\Gamma$.
 In case (iii) 
 consider  an $x \in V(\Gamma)$ such that
  $2x \not\in \{s_1, s_1-s_2,s_1+s_2,0\}$; again such an $x \in V(\Gamma)$ exists by Fact~\ref{solutions 2x=g}
and as $|\Gamma|=n/2 >4 \cdot 2^r$. Then $x$ has 
 degree at least $4$ in $\Gamma$ since $x+s_1$, $x+s_2$, $x-s_2$ and $s_1-x$ are distinct neighbours of $x$.
 In cases (iv) and (vi) 
 consider  an $x \in V(\Gamma)$ such that
 $2x \not\in \{0,s_1-s_2,s_2-s_1,  s_1, s_2\}$; then $x+s_1$, $x+s_2$, $s_1-x$ and $s_2-x$ are distinct neighbours of $x$.
 In case (v) consider an $x \in V(\Gamma)$ such that
  $2x\not\in \{0,-2s,2s, s, -s\}$; then $x+s$, $x-s$, $s-x$ and $-s-x$ are distinct neighbours. Finally, in case (vii)  consider  an $x \in V(\Gamma)$ such that $2x\not\in \{2s,0,s,-s\}$; then $x+s$, $x-s$, $-x$ and $s-x$ are distinct neighbours of $x$.

  \smallskip
  
Now we have proved that $\Delta \geq 4$, we can complete the proof of the claim.
By Claim~\ref{nb edges distinct link graph} we have $e(\Gamma)\geq \frac{\Delta n}{4}-|S|(|S\cup -S|+1)2^r$. Applying Lemma~\ref{stability} with $k:=\lceil{{(\Delta-2)n}/{4}-|S|(|S\cup -S|+1)2^r} \rceil $, we have
$$\mis(\Gamma)\leq 3^{\frac{\Delta}{13}}\cdot 3^{\frac{n}{6}-\frac{(\Delta-2)n}{52\Delta}+\frac{|S|(|S\cup -S|+1)\cdot2^r}{13\Delta}}.$$
Since $|S|(|S\cup -S|+1)\leq 87^2\cdot(2\cdot 87^2+1)$ and $4\leq\Delta\leq3\cdot87^2$ by Claim~\ref{degree distinct link graph}, we have
$$\mis(\Gamma)\leq 3^{\frac{3\cdot87^2}{13}}\cdot 3^{\left(\frac{49}{312}\frac{n}{2^r}+\frac{87^2\cdot(2\cdot 87^2+1)}{52}\right) \cdot2^r}.$$
The choice of $\eps$ and $C$ implies that $\mis(\Gamma)\leq2^{(1/4-\eps)n-2^{r-2}}$, since $n/2^r\geq C$.\footnote{This is the  place in the argument where we use the precise choice of $C$.}
\end{proof}
The next claim will be used to bound
$f_{\max}^{\star,4}(G)$.
\begin{claim}\label{type 4}
    Let $S=\{s\}$ with $s\ne 0$. 
     Then $\mis(\Gamma)\leq 2^{(1/4-\eps)n-2^{r-2}}$.
\end{claim}
\begin{proof}
    Let $\ell$ denote the order of $s$. We first handle the case when $\ell \in\{2,3\}$.
    \begin{enumerate}[label=(\roman*)]
    \item $\ell=2$. In this case, if $x \in V(\Gamma)=A$ is such that $2x\not \in\{0,s\}$, then the component of $\Gamma $ containing $x$ is precisely a 
     $4$-cycle $\{x,x+s,-x,s-x\}$.
     If $x \in V(\Gamma)$ is such that $2x\in\{0,s\}$,
    the  component containing $x$ is  an edge $\{x,x+s\}$. 
    Let $I_2:=\{x\in V(\Gamma) \ | \ 2x\in\{0,s\}\}$. By Fact~\ref{solutions 2x=g}, $|I_2|\leq 2^{r+1}$. Then, $\mis(\Gamma)\leq 2^{\frac{|A|-|I_2|}{4}}\cdot 2^{|I_2|/2}\leq 2^{n/8+2^{r-1}}$. The choice of $\eps$ and $C$ implies that $\mis(\Gamma)\leq 2^{(1/4-\eps)n-2^{r-2}}$, since $n/2^r\geq C$.
    
    \item $\ell=3$. In this case, 
    if $x \in V(\Gamma)$ is such that $2x\not \in\{0,s, 2s\}$, then the component of $\Gamma $ containing $x$ is precisely a 
     $C_3\square K_2$;  a perfect matching between the triangles $\{x,x+s,x+2s\}$ and $\{-x+s,-x,-x+2s\}$. If $x \in V(\Gamma)$ such that $2x\in\{0,s,2s\}$, then the component of $\Gamma $ containing $x$ is a triangle $\{x,x+s,x+2s\}$. Let $I_3:=\{x\in V(\Gamma) \ | \ 2x\in\{0,s,2s\}\}$. By Fact~\ref{solutions 2x=g}, $|I_3|\leq 3\cdot2^r$. Noting that $C_3\square K_2$ contains precisely $6$ maximal independent sets, we have that
     $$\mis(\Gamma)\leq 6^{\frac{|A|-|I_3|}{6}}\cdot 3^{|I_3|/3}\leq 6^{n/12-2^{r-1}}\cdot3^{2^r}.$$
    The choice of $\eps$ and $C$ implies that $\mis(\Gamma)\leq2^{(1/4-\eps)n-2^{r-2}}$, since $n/2^r\geq C$.
\end{enumerate}
Now  assume $\ell>3$. Then we can adapt \cite[Claim 3.9]{ls} to our distinct sum-free set setting: there exists a subset $A_t\subset A=V(\Gamma)$ with $|A_t|\leq 2^r$ that intersects all triangles in $\Gamma$. The proof is exactly the same as in the non-distinct setting (see \cite[Claim 3.9]{ls}), so we omit it.
Let $A_t$ be such a set, and let $\Gamma':=\Gamma\setminus  A_t$. By Claim \ref{degree distinct link graph}, every vertex in $\Gamma$ has degree at most 3, and by Claim \ref{nb edges distinct link graph}, $e(\Gamma)\geq \frac{3|A|}{2}-3\cdot2^r$, and thus, 
$$e(\Gamma')\geq e(\Gamma)-3|A_t|\geq \frac{3|A|}{2}-6\cdot 2^r.$$
By Lemma \ref{almosttrianglefree}, with $D:=3$ and $k:=|A|-6\cdot 2^r$, we obtain
$$\mis(\Gamma)\leq 2^{\frac{n}{4}-\frac{n/2-6\cdot2^r}{900}+2\cdot2^r}=2^{\frac{449}{1800}n+\frac{301}{150}\cdot2^r}.$$
The choice of $\eps$ and $C$ implies that $\mis(\Gamma)\leq2^{(1/4-\eps)n-2^{r-2}}$, since $n/2^r\geq C$.
\end{proof}
With the previous claims at hand, it is now straightforward to bound the number
of type 0--4 maximal distinct sum-free subsets of $G$.
\begin{claim}\label{upper bound type 0-4}
    $f_{\max}^{\star,0}(G)+f_{\max}^{\star,1}(G)+f_{\max}^{\star,2}(G)+f_{\max}^{\star,3}(G)+f_{\max}^{\star,4}(G)\leq \eta\cdot 2^{n/4-2^{r-2}}$.
\end{claim}
\begin{proof}
    We already saw that $f_{\max}^{\star,0}(G)\leq 2^{\delta n}$.  By Lemma~\ref{distinct containers}, there are at most $2^{\delta n}$ containers $F\in \mathcal F$. For each choice of $F$, there are at most $2^{\delta n}$ choices for $S\subseteq B$. 
    For any $S$ as in the definitions of types 1--4,
    Claims~\ref{type 1}, \ref{type 2}, \ref{type 3} and \ref{type 4} imply that $\mis(\Gamma)\leq 2^{(1/4-\eps)n-2^{r-2}}$. Thus, by Lemma~\ref{extension link graph},
    $$f_{\max}^{\star,i}(G)\leq 2^{\delta n}\cdot 2^{\delta n}\cdot  2^{(1/4-\eps)n-2^{r-2}}=2^{(1/4+2\delta-\eps)n-2^{r-2}}$$ for each $i \in [4]$. As $\delta \ll \eps \ll \eta$, the claim now follows.

\end{proof}
The next two claims provide an upper bound on the number of type 5 and 6 maximal distinct sum-free subsets of $G$.
\begin{claim}\label{upper bound type 5}
    $f_{\max}^{\star,5}(G)\leq (2^{r_1}-1)2^{n/4}+(2^r-2^{r_1})2^{n/4-2^{r-2}}$.
\end{claim}
\begin{proof}
    Recall we  assume that each relevant container is of the form $F=A\cup B$ where $A=\Tilde{A}\oplus K$ with $\Tilde{A}$ a coset (not containing the zero element) associated to a subgroup of index 2 in $\Z_{2^{\alpha_1}}\oplus\ldots\oplus\Z_{2^{\alpha_{r_1}}}\oplus\Z_2^{r_2}$.
    Moreover, for type 5 maximal distinct sum-free sets, $S=\{0\}$; so to upper bound $f_{\max}^{\star,5}(G)$, it suffices to sum up the number of maximal independent sets in all distinct link graphs of the form $L_{\{0\}}^\star[\Tilde{A}\oplus K]$. 
    Let $\Tilde{A}$ be a coset (not containing the zero element) associated to a subgroup of index 2 in $\Z_{2^{\alpha_1}}\oplus\ldots\oplus\Z_{2^{\alpha_{r_1}}}\oplus\Z_2^{r_2}$, and set $\Gamma:=L_{\{0\}}^\star[\Tilde{A}\oplus K]$. We distinguish two cases depending on the structure of $\Tilde{A}$.
    \begin{enumerate}[label=(\roman*)]
    \item $\Tilde{A}$ is a coset associated to a subgroup of index 2 with rank $r-1$. In this case, $\Gamma$ consists of
    a set of isolated vertices $I:=\{x\in A \ | \ 2x=0\}$ together with 
    a matching where edges are of the form $\{x, -x\}$ for $2x\ne0$. Since $\Tilde{A}$ is a coset associated to a subgroup of index 2 with rank $r-1$, we have $|I|= 2^{r-1}$. Thus, $\mis(\Gamma)= 2^{\frac{|A|-|I|}{2}}= 2^{n/4-2^{r-2}}$. By Lemma \ref{nb subgroups of index 2 with fixed rank}, there are $2^r-2^{r_1}$ such sets $\Tilde{A}$; thus such link graphs give rise to at most $(2^r-2^{r_1})2^{n/4-2^{r-2}}$ type~5 maximal distinct sum-free sets.
    
    \item $\Tilde{A}$ is a coset associated to a subgroup of index 2 with rank $r$. In this case, there are no elements in $A=\Tilde{A}\oplus K$ such that $2x=0$. Then, $\Gamma$ is a perfect matching and $\mis(\Gamma)=2^{n/4}$. By Lemma \ref{nb subgroups of index 2 with fixed rank}, there are $2^{r_1}-1$ such sets $\Tilde{A}$;  thus such link graphs give rise to at most  $(2^{r_1}-1)2^{n/4}$ type 5 maximal distinct sum-free sets.
\end{enumerate}
Summing over all the possible sets $A$, we obtain $f_{\max}^{\star,5}(G)\leq (2^{r_1}-1)2^{n/4}+(2^r-2^{r_1})2^{n/4-2^{r-2}}$.
\end{proof}

\begin{claim}\label{upper bound type 6}
    $f_{\max}^{\star,6}(G)\leq (2^{2r-1}+2^{r+r_1-1}-2^{r+1}+1)2^{n/4}$.
\end{claim}
\begin{proof}
    As in Claim~\ref{upper bound type 5}, we  assume that each relevant container is of the form $F=A\cup B$ where $A=\Tilde{A}\oplus K$ with $\Tilde{A}$ a coset (not containing the zero element)  associated to a subgroup of index 2 in $\Z_{2^{\alpha_1}}\oplus\ldots\oplus\Z_{2^{\alpha_{r_1}}}\oplus\Z_2^{r_2}$, and where $A$ and $B$ are disjoint. 
    Moreover, for type 6 maximal distinct sum-free sets, $S=\{0,s\}$ for some order 2 element $s\in B$; so to upper bound $f_{\max}^{\star,6}(G)$, it suffices to sum up the number of maximal independent sets in all distinct link graphs of the form $L_{\{0,s\}}^\star[\Tilde{A}\oplus K]$ where $s \notin \Tilde{A}\oplus K$ has order $2$. 
    Let $\Tilde{A}$ be a coset (not containing the zero element) associated to a subgroup of index 2 in $\Z_{2^{\alpha_1}}\oplus\ldots\oplus\Z_{2^{\alpha_{r_1}}}\oplus\Z_2^{r_2}$, let $s\notin \Tilde{A}\oplus K$ be an order 2 element, and set $\Gamma:=L_{\{0,s\}}^\star[\Tilde{A}\oplus K]$. 
    We saw in the proof of Proposition~\ref{fms even groups2c} that $\mis(\Gamma)=2^{n/4}$. Now let us count the pairs $(\Tilde{A},\{0,s\})$.
    \begin{enumerate}[label=(\roman*)]
        \item $\Tilde{A}$ is a coset associated to a subgroup of index 2 with rank $r-1$. By Lemma \ref{nb subgroups of index 2 with fixed rank}, there are $2^r-2^{r_1}$ such sets $\Tilde{A}$. Since there are $2^{r-1}-1$ order 2 elements of $G$ not belonging to $A=\Tilde{A}\oplus K$, there are $2^{r-1}-1$ choices for $s$. Thus, such pairs $(\Tilde{A},\{0,s\})$ contribute  at most $(2^r-2^{r_1})(2^{r-1}-1)2^{n/4}$ type 6 maximal distinct sum-free sets.
        
        \item $\Tilde{A}$ is a coset associated to a subgroup of index 2 with rank $r$. By 
        Lemma~\ref{nb subgroups of index 2 with fixed rank}, there are $2^{r_1}-1$ such sets $\Tilde{A}$. Since $A=\Tilde{A}\oplus K$ contains no order 2 elements, there are $2^r-1$ choices for $s$. Thus, such pairs $(\Tilde{A},\{0,s\})$ contribute  at most $(2^{r_1}-1)(2^r-1)2^{n/4}$ type 6 maximal distinct sum-free sets.
    \end{enumerate}
Summing over all the possible choices, we obtain 
\begin{align*}
    f_{\max}^{\star,6}(G)&\leq \left[(2^r-2^{r_1})(2^{r-1}-1)+(2^{r_1}-1)(2^r-1)\right]2^{n/4}\\
    &=(2^{2r-1}+2^{r+r_1-1}-2^{r+1}+1)2^{n/4}.
\end{align*}
\end{proof}

Since every maximal distinct sum-free subset of $G$ is of type 0--6, Claims~\ref{upper bound type 0-4}--\ref{upper bound type 6} yield the desired upper bound:
$$\fms(G)\leq (2^{2r-1}+2^{r+r_1-1}-2^{r+1}+2^{r_1})2^{n/4}+(2^r-2^{r_1}+\eta)2^{n/4-2^{r-2}}.$$ \qed

\subsection{Proof of Theorem~\ref{fms even groups2b}}
 Let $C:=3235084117$. Given any $\eta >0$, define additional constants $ \eps, n_0 >0$ so that
$$0< 1/n_0  \ll \eps \ll \eta, 1/C.$$

Let $G$ be an abelian group of even order $n \geq n_0$ as in the statement of the theorem. 
Thus,
$G=\Z_{2^{\alpha_1}}\oplus\ldots\oplus\Z_{2^{\alpha_{r_1}}}\oplus\Z_2^{r_2}\oplus K$ with $r_1$ and $r_2$ non-negative integers, $\alpha_1\geq\ldots\geq\alpha_{r_1}\geq2$,  and 
$|K|$ odd. Furthermore, $n/2^{r_1+r_2}\geq C$.

 For a set $A:=\Tilde{A}\oplus K$ with $\Tilde{A}$ a coset (not containing the zero element) associated to a subgroup of index 2 in $\Z_{2^{\alpha_1}}\oplus\ldots\oplus\Z_{2^{\alpha_{r_1}}}\oplus\Z_2^{r_2}$, and a set $S\subseteq G\setminus A$, we say that the pair $(A,S)$ \textit{generates} a distinct sum-free set $D$ in $G$ if there exists a maximal independent set $I$ in the link graph $L_S^\star[A]$ such that $D=I\cup S$.

\begin{claim}\label{mdsfs generated by B}
    Consider $A=\Tilde{A}\oplus K$ with $\Tilde{A}$ a coset (not containing the zero element) of a subgroup $H$ of index 2 in $\Z_{2^{\alpha_1}}\oplus\ldots\oplus\Z_{2^{\alpha_{r_1}}}\oplus\Z_2^{r_2}$. 
    \begin{enumerate}[label=(\roman*)]
        \item If $H$ has rank $r-1$, then $(A,\{0\})$ generates at least $2^{n/4-2^{r-2}}-(n/2-2^{r-1})\cdot2^{(1/4-\eps)n-2^{r-2}}$ maximal distinct sum-free sets in $G$.
        \item If $H$ has rank $r$, then $(A,\{0\})$ generates at least $2^{n/4}-(n/2-2^r)\cdot2^{(1/4-\eps)n-2^{r-2}}$ maximal distinct sum-free sets in $G$.
        \item If $s\in G\setminus(A\cup\{0\})$ has order 2, $(A,\{0,s\})$ generates at least $2^{n/4}-(n/2-2)\cdot2^{(1/4-\eps)n-2^{r-2}}$ maximal distinct sum-free sets in $G$.
    \end{enumerate}
\end{claim}
\begin{proof}
   (i) We fix such an $A$. In the proof of Claim~\ref{upper bound type 5} we saw that $L^\star_{\{0\}}[A]$ has $2^{n/4-2^{r-2}}$ maximal independent sets. Let $I\subseteq A$ be such a maximal independent set and suppose that $I\cup \{0\}$ is not a maximal distinct sum-free set in $G$; call such an $I$ \textit{bad}. Then there exists $s\in G\setminus (A\cup \{0\})$ such that $\{0,s\}\cup I$ is distinct sum-free. So $I$ is a maximal independent set in $L^\star_{\{0,s\}}[A]$. 
   
   Suppose first that $s$ does not have order 2.
   By case (vii) in the proof of Claim~\ref{type 3}, there are at most $2^{(1/4-\eps)n-2^{r-2}}$ such sets $I$.
   
   Now suppose that $s$ has order 2. We saw in Claim~\ref{upper bound type 5} that $L^\star_{\{0\}}[A]$ consists of a  matching between $|A|-2^{r-1}$ elements, and has $2^{r-1}$ isolated vertices. So $|I|=(|A|-2^{r-1})/2+2^{r-1}=(|A|+2^{r-1})/2$. 
   On the other hand, in the proof of Proposition~\ref{fms even groups2c}, we saw that $L^\star_{\{0,s\}}[A]$ consists of the disjoint union of $(|A|-2^{r-1}-a(s))/4$ copies of $K_4$ and $(2^{r-1}+a(s))/2$ copies of $K_2$, where $a(s)$ is the number of $x\in A$ such that $2x=s$. So $|I|=(|A|+2^{r-1}+a(s))/4$, which is possible only if $a(s)=|A|+2^{r-1}=n/2+2^{r-1}$. 
   Since  $n/2> n/C\geq 2^r$,  such an $I$ cannot exist since $a(s)\leq 2^r$ by Fact~\ref{solutions 2x=g}. 

Thus, all bad $I$s come from  $s$ that do not have order 2. The bound in the claim now follows as there are $n/2-2^{r-1}$ choices for $s \in G \setminus (A\cup\{0\})$ that do not have order 2.

   \smallskip
   
   (ii) 
   In the proof of Claim~\ref{upper bound type 5} 
   we saw that $L^\star_{\{0\}}[A]$ has $2^{n/4}$ maximal  independent sets $I$. In this case, again 
   all bad $I$s come from  $s$ that do not have order 2. Indeed, suppose $s \in G \setminus A$ has order 2. In the proof of
 Claim~\ref{upper bound type 5} we saw  maximal independent sets in $L^\star_{\{0\}}[A]$ have size $|A|/2$. Meanwhile, 
 in the proof of Proposition~\ref{fms even groups2c} we saw that maximal independent sets in $L^\star_{\{0,s\}}[A]$ have size $|A|/4+a(s)/4$. These sizes coincide only if $a(s)=n/2 > n/C \geq 2^r$, which is impossible by Fact~\ref{solutions 2x=g}. The bound in the claim now follows by case (vii) from the proof of Claim~\ref{type 3} and the fact that there are $n/2-2^r$ elements in $G\setminus (A\cup\{0\})$ that do not have order 2.

   \smallskip 
   
   (iii) We saw in the proof of Proposition~\ref{fms even groups2c} that $L^\star_{\{0,s\}}[A]$ has $2^{n/4}$ maximal independent sets $I$. If $I\cup \{0,s\}$ is not a maximal sum-free subset of $G$, then there  exists $s'\in G\setminus(A\cup\{0,s\})$ such that $I$ is a maximal independent set in $L^\star_{\{0,s,s'\}}[A]$. By case (iv) in the proof of Claim \ref{type 3}, there are at most $2^{(1/4-\eps)n-2^{r-2}}$ such sets $I$. The bound in the claim now follows as there are $n/2-2$ choices for $s'\in G\setminus(A\cup\{0,s\})$.
\end{proof}

Now we need to be careful because we may count the same maximal distinct sum-free set in two different pairs $(A,\{0\})$ and $(A',\{0\})$.

\begin{claim}\label{mdsfs generated by both A and A'}
    Consider distinct sets $A=\Tilde{A}\oplus K$ and $A'=\Tilde{A'}\oplus K$, with $\Tilde{A}$ and $\Tilde{A'}$ cosets of  subgroups of index 2 in $\Z_{2^{\alpha_1}}\oplus\ldots\oplus\Z_{2^{\alpha_{r_1}}}\oplus\Z_2^{r_2}$. 
    Then there is at most one maximal distinct sum-free subset of $G$ generated by both $(A,\{0\})$ and $(A',\{0\})$.
\end{claim}
\begin{proof}
    Note that $|A\cap A'|=|\Tilde{A}\cap\Tilde{A'}|\leq \mu(G)/2$. Let $D$ be a maximal distinct sum-free subset of $G$ and suppose that it is generated by both  $(A,\{0\})$ and $(A',\{0\})$. Then $D=\{0\}\cup I$ with $I\subseteq A\cap A'$.
    \begin{enumerate}[label=(\roman*)]
        \item If $\Tilde{A}$ and $\Tilde{A}'$ are both cosets associated to subgroups of index 2 with rank $r-1$, then $|I|=\mu(G)/2+2^{r-2}$. Indeed, in this case, in the proof of Claim~\ref{upper bound type 5} we saw that the link graph $L_{\{0\}}^\star[A]$ consists of a  matching containing $\mu(G)-2^{r-1}$ vertices, as well as $2^{r-1}$ isolated vertices. So the size of a maximal independent set is $(\mu(G)-2^{r-1})/2+2^{r-1}=\mu(G)/2+2^{r-2}$. 
        In particular $|A\cap A'|\geq |I| \geq \mu(G)/2+2^{r-2}$, which is impossible.
        
        \item If $\Tilde{A}$ and $\Tilde{A}'$ are both  cosets associated to  subgroups of index 2 with rank $r$, then $|I|=\mu(G)/2$. Indeed we saw in the proof of Claim~\ref{upper bound type 5} that $L_{\{0\}}^\star[A]$ is a perfect matching. This is only possible if $I=A \cap A'$.
So there is at most one possible set $D$ generated by both.
        
        \item If we have one of each, then $|I|=\mu(G)/2=\mu(G)/2+2^{r-2}$, which is impossible.
    \end{enumerate}
\end{proof}

\begin{claim}\label{mdsfs generated by both A and (A',{0,s})}
    Consider two sets $A=\Tilde{A}\oplus K$ and $A'=\Tilde{A'}\oplus K$, with $\Tilde{A}$ and $\Tilde{A'}$ cosets of  subgroups of index 2 in $\Z_{2^{\alpha_1}}\oplus\ldots\oplus\Z_{2^{\alpha_{r_1}}}\oplus\Z_2^{r_2}$, and $s\in G\setminus (A'\cup\{0\})$ with order 2.
    Then there is no maximal distinct sum-free set generated by both $(A,\{0\})$ and $(A',\{0,s\})$.
\end{claim}
\begin{proof}
    Let $D$ be a maximal distinct sum-free set and suppose for a contradiction that it is generated by both  $(A,\{0\})$ and $(A',\{0,s\})$. Then $D=\{0\}\cup I=\{0,s\}\cup I'$ with $I\subseteq A$ and $I'\subseteq A'$. In particular, $|I|=|I'|+1$, and $s\in I\subseteq A$; so $A$ contains an order 2 element, and thus it is associated to a subgroup with rank $r-1$. In the proof of Claim~\ref{upper bound type 5} we saw that in this case we have $|I|=\mu(G)/2+2^{r-2}$.
    \begin{enumerate}[label=(\roman*)]
        \item If $\Tilde{A}'$ is a coset associated to a subgroup of index 2 with 
        rank $r-1$, then by the proof of 
        Proposition~\ref{fms even groups2c}, $|I'|=(\mu(G)+2^{r-1}+a(s))/4$. Since $|I|=|I'|+1$, this implies that  $a(s)=n/2+2^{r-1}-4> 2^r$, which is impossible by Fact~\ref{solutions 2x=g}.
        
        \item If $\Tilde{A}'$ is a coset associated to a subgroup of index 2 with 
        rank $r$, then by the proof of Proposition~\ref{fms even groups2c}, $|I'|=(\mu(G)+a(s))/4$. It follows that $a(s)=n/2+2^r-4> 2^r$, which is again impossible.
    \end{enumerate}
\end{proof}

\begin{claim}\label{mdsfs generated by both (A,{0,s}) and (A',{0,s'})}
    Consider two sets $A=\Tilde{A}\oplus K$ and $A'=\Tilde{A'}\oplus K$, with $\Tilde{A}$ and $\Tilde{A'}$ cosets of  subgroups of index 2 in $\Z_{2^{\alpha_1}}\oplus\ldots\oplus\Z_{2^{\alpha_{r_1}}}\oplus\Z_2^{r_2}$, and $s\in G\setminus (A\cup\{0\})$, $s'\in G\setminus (A'\cup\{0\})$ both with order 2.
    Then there are at most $3^{n/12}$ maximal distinct sum-free sets generated by both $(A,\{0,s\})$ and $(A',\{0,s'\})$.
\end{claim}

\begin{proof}
    Let $D$ be a maximal distinct sum-free set generated by both $(A,\{0,s\})$ and $(A',\{0,s'\})$. 
    Then $D=\{0,s\}\cup I=\{0,s'\}\cup I'$ with $I\subseteq A$ and $I'\subseteq A'$.
    Let $D':=D\setminus\{0,s,s'\}$. Then by Lemma~\ref{extension link graph}, $D'$ is a maximal independent set in $L_{\{0,s,s'\}}[A\cap A']$. Since $|A\cap A'|\leq n/4$, by (\ref{Moon-Moser}) there are at most $3^{n/12}$ choices for $D'$, and thus for $D$, as desired.
\end{proof}

 Lemma~\ref{nb subgroups of index 2 with fixed rank} implies that there are $2^r-2^{r_1}$ choices for $(A,\{0\})$ in Claim~\ref{mdsfs generated by B}(i) and 
 $2^{r_1}-1$ choices for $(A,\{0\})$ in Claim~\ref{mdsfs generated by B}(ii).
The argument in the proof of Claim~\ref{upper bound type 6} implies that there are 
$2^{2r-1}+2^{r+r_1-1}-2^{r+1}+1$ choices for $(A,\{0,s\})$ in Claim~\ref{mdsfs generated by B}(iii). 
By Lemma~\ref{nb subgroups of index 2 with fixed rank} there are $2^r-1$ subgroups of index $2$ of 
$\Z_{2^{\alpha_1}}\oplus\ldots\oplus\Z_{2^{\alpha_{r_1}}}\oplus\Z_2^{r_2}$. Thus, there are at most
 $\binom{2^r-1}{2}$ choices for $(A,A')$
 in each of Claims~\ref{mdsfs generated by both A and A'}--\ref{mdsfs generated by both (A,{0,s}) and (A',{0,s'})}.
 Further, there
  at most $(2^r)^4$ choices for $(A,\{0,s\})$ and $(A',\{0,s'\})$ in Claim~\ref{mdsfs generated by both (A,{0,s}) and (A',{0,s'})}. Therefore,
  Claims~\ref{mdsfs generated by B}--\ref{mdsfs generated by both (A,{0,s}) and (A',{0,s'})}  imply that 
\begin{align*}
    \fms(G)&\geq (2^{r_1}-1)\left(2^{n/4}-(n/2-2^r)\cdot2^{(1/4-\eps)n-2^{r-2}}\right) \\
    & + (2^r-2^{r_1})\left(2^{n/4-2^{r-2}}-(n/2-2^{r-1})\cdot2^{(1/4-\eps)n-2^{r-2}}\right)\\
    & +(2^{2r-1}+2^{r+r_1-1}-2^{r+1}+1)\left(2^{n/4}-(n/2-2)\cdot2^{(1/4-\eps)n-2^{r-2}}\right)\\
    &- \binom{2^r-1}{2} - (2^r)^4\cdot 3^{n/12}.
\end{align*}
As $n$ is sufficiently large, we have that
$$\fms(G)\geq (2^{2r-1}+2^{r+r_1-1}-2^{r+1}+2^{r_1})2^{n/4}+(2^r-2^{r_1}-\eta)2^{n/4-2^{r-2}}.$$
\qed

\begin{remark}
    Note that we did not try to optimise the value of $C$ in the proof of Theorem~\ref{fms even groups2b}. With more careful analysis of the link graphs, it may be possible to reduce the value of $C$ to a single-digit number. Theorem~\ref{fms even groups2} should also hold with a much smaller choice of $C$, although with our argument the choice of $C$ is essentially tight. 
\end{remark}

\section{A conjecture on maximal independent sets in graphs with perfect matchings}\label{sec:conj}

We conclude this paper by raising a conjecture on the number of maximal independent sets in graphs that contain a perfect matching. 

\begin{conjecture}\label{conj3}
Let $\Gamma$ be an $n$-vertex graph that contains a perfect matching. Then $\Gamma$ contains at most $2^{n/2}$ maximal independent sets.
\end{conjecture}

Note that we came to this conjecture whilst studying auxiliary (link) graphs for Conjecture~\ref{conj2}. A resolution of Conjecture~\ref{conj3} may therefore be helpful for attacking Conjecture~\ref{conj2}.
Two examples show that the bound in Conjecture~\ref{conj3} cannot be lowered. Indeed, if $\Gamma$ itself is a perfect matching then $\Gamma$ contains precisely $2^{n/2}$ maximal independent sets. Similarly, suppose $\Gamma$ is the disjoint union of $n/6$ copies of the following graph $H$: $H$ consists of two disjoint triangles joined together by a single edge. Then $H$ contains $8$ maximal independent sets, and thus $\Gamma$ itself contains precisely $8^{n/6}=2^{n/2}$ maximal independent sets.

{\noindent \bf Additional note.} Since we submitted this paper, Balogh, Garcia, Liu and Yang~\cite{bgly} have proven Conjecture~\ref{conj3} in a  strengthened form; see~\cite[Theorem 1.9]{bgly}. In particular, the resolution 
of  Conjecture~\ref{conj3} is a crucial ingredient to their resolution (in the negative) of Question~\ref{quesX}, as well as their proof of Conjecture~\ref{conj2}
in the case of even order groups $G \not \cong \Z ^k_2 \oplus \Z_3$.

\section*{Acknowledgments}
The authors are grateful to Vsevolod Lev and Wojciech Samotij for discussions relating to Conjecture~\ref{conj5}. This research was partially carried out whilst the second author was visiting the Universit\'e Bourgogne Europe. This author is grateful for the hospitality he received. The first author thanks Rafik Souanef for pointing out Lemma \ref{nb subgroups of index 2 with fixed rank}.
The authors are also grateful to the referee for their careful review.

\smallskip

{\noindent \bf Data availability statement.}
There are no additional data beyond that contained within the main manuscript.


\begin{thebibliography}{99}
\bibitem{alon} N. Alon, J. Balogh, R. Morris and W. Samotij, Counting sum-free sets in abelian groups, \emph{Israel J. Math.}, 199, (2014), 309--344. 

\bibitem{bgly} J. Balogh, R.I. Garcia, H. Liu and N. Yang,
Infinitely many groups exhibiting intermediate growth in maximal sum-free sets,  	arXiv:2509.19248.

\bibitem{BLST}
 J.~Balogh, H.~Liu, M.~Sharifzadeh and A.~Treglown,
 \newblock The number of maximal sum-free subsets of integers,
\newblock  {\em  Proc. Amer. Math. Soc.},  143, (2015), 4713--4721.

\bibitem{BLST2}
 J.~Balogh, H.~Liu, M.~Sharifzadeh and A.~Treglown,
 \newblock Sharp bound on the number of maximal sum-free subsets of integers,
\newblock  {\em  J. Euro. Math, Soc.}, 20, (2018), 1885--1911.

\bibitem{b2} J. Balogh, R. Morris and W. Samotij,  Random sum-free subsets of abelian groups, \emph{Isr. J. Math.}, 199, (2014),
651--685.

\bibitem{bms} J.~Balogh, R.~Morris and W.~Samotij, Independent sets in hypergraphs, \emph{J. Amer. Math. Soc.}, 28, (2015), 669--709.

\bibitem{bush} N. Bushaw, M. Collares Neto, R. Morris and P. Smith, The sharp threshold for maximum-size sum-free subsets in even-order abelian groups, \emph{Combin. Probab. Comput.}, 24, (2015), 609--640.


\bibitem{conlonsurvey} D. Conlon, Combinatorial theorems relative to a random set, Proceedings of the International Congress of Mathematicians 2014, Vol. 4, 303--328.
\bibitem{conlongowers} D. Conlon and W.T. Gowers, Combinatorial theorems in sparse random sets, \emph{Ann. Math.}, 84, (2016), 367--454.

\bibitem{yan} P.H. Diananda and H.P. Yap, Maximal sum-free sets of elements of finite groups, \emph{Proc. Japan Acad.}, {45}, (1969), 1--5.




\bibitem{G-R}
B.~Green,
\newblock A {S}zemer\'edi-type regularity lemma in abelian groups, with
              applications,
\newblock {\em Geom. Funct. Anal.}, 15, (2005), 340--376. 



\bibitem{GR-g}
B. Green and I. Ruzsa,
\newblock
Sum-free sets in abelian groups,
\newblock {\em Israel J. Math.},  147, (2005), 157--189.   


\bibitem{HT} N. Hassler and A. Treglown,
On maximal sum-free sets in abelian groups,
\emph{Electron. J. Combin.}, 29, (2022), \#P2.32.
\bibitem{HT2}
M.~Hujter and Z.~Tuza,
\newblock The number of maximal independent sets in triangle-free
              graphs,
\newblock {\em SIAM J. Discrete Math.}, 6, (1993), 284--288.

\bibitem{lev} V.F. Lev, Large sum-free sets in $\mathbb Z^n_5$, \emph{J. Combin. Theory Ser. A}, 205, (2024), 105865.

\bibitem{ls} H. Liu and M. Sharifzadeh, Groups with few maximal sum-free sets, \emph{J. Combin. Theory Ser. A}, 177, (2021), 105333.
\bibitem{mallet-burgues} L. Mallet-Burgues, Arithmétique des Groupes Abéliens Finis, arXiv:2305.01987.
\bibitem{MM}
J.W.~Moon and L.~Moser,
\newblock On cliques in graphs,
\newblock {\em Israel J. Math.}, 3, (1965), 23--28.


\bibitem{RS} C. Reiher and S. Zotova, Large sum-free sets in finite vector spaces I, arXiv:2408.11232.
\bibitem{saxton} D. Saxton and A. Thomason, Hypergraph containers, \emph{Invent. Math.}, 201, (2015), 925--992.
\bibitem{schacht} M. Schacht, Extremal results for random discrete structures, \emph{Ann. Math.}, 184, (2016), 331--363.
\bibitem{ve} L. Versteegen, The structure of large sum-free sets in $\mathbb F^n_p$, \emph{Q. J. Math.}, 75, (2024), 1057--1071.

\bibitem{wolf} G. Wolfovitz,
\newblock Bounds on the number of maximal sum-free sets,
\newblock {\em Euro. J. Combin.}, 30, (2009), 1718--1723.



\end{thebibliography}
\end{document}